\RequirePackage{fix-cm}

\documentclass[smallextended,envcountsect,envcountsame,numbook,nospthms]{svjour3} 
\smartqed 

\makeatletter 
\def\cl@chapter{}
\makeatother
\usepackage[utf8]{inputenc}
\usepackage{lmodern}
\usepackage[pdfpagelabels,colorlinks]{hyperref}
\usepackage{mathtools,amsfonts,amssymb,amscd,dsfont}
\usepackage[misc]{ifsym}
\usepackage[shortlabels]{enumitem}
\usepackage{lmodern}
\usepackage{amssymb}
\usepackage{subcaption}
\usepackage{cite}
\usepackage{mathrsfs}
\usepackage[separate-uncertainty = true]{siunitx}
\usepackage{booktabs}
\usepackage[pdftex,dvipsnames]{xcolor}
\usepackage[colorinlistoftodos,prependcaption,textsize=tiny]{todonotes}
\reversemarginpar 

\journalname{Optimization Letters}

\usepackage{amsthm}
\newtheorem{theorem}{Theorem}[section]
\newtheorem{lemma}[theorem]{Lemma}
\newtheorem{corollary}[theorem]{Corollary}
\newtheorem{proposition}[theorem]{Proposition}
\theoremstyle{definition}
\newtheorem{definition}[theorem]{Definition}

\newtheorem{fact}[theorem]{Fact}
\theoremstyle{remark}

\newtheorem{example}[theorem]{Example}

\usepackage[capitalize,nameinlink]{cleveref}
\usepackage{autonum}

\Crefname{lemma}{Lemma}{Lemmas}
\Crefname{proposition}{Proposition}{Propositions}
\Crefname{definition}{Definition}{Definitions}
\Crefname{proposition}{Proposition}{Propositions}
\Crefname{corollary}{Corollary}{Corollaries}
\Crefname{fact}{Fact}{Facts}

\newcommand{\RR}{\ensuremath{\mathds R}}

\def\Kpolar{K^{\circ}}
\newcommand{\Tcone}{\ensuremath{\mathcal T}}

\DeclareMathOperator{\aff}{aff}

\DeclareMathOperator{\inte}{int}

\DeclareMathOperator{\closu}{cl}

\DeclareMathOperator{\cone}{cone}

\DeclareMathOperator*{\argmin}{arg\,min}
\newcommand{\norm}[1]{\left\lVert#1\right\rVert}
\newcommand{\scal}[2]{\left\langle{#1},{#2}  \right\rangle}
\DeclareMathOperator{\circum}{circ}

\usepackage[shortlabels]{enumitem}
\newlist{lista}{enumerate}{1}
\setlist[lista]{label=\alph*., nosep,leftmargin=*,align=right}

\newlist{listi}{enumerate}{1}
\setlist[listi]{label={\upshape(\roman*\upshape)},leftmargin=*,align=right, widest=iii,nosep, format=\bf}

\allowdisplaybreaks

\crefformat{equation}{\textup{#2(#1)#3}}
\crefrangeformat{equation}{\textup{#3(#1)#4--#5(#2)#6}}
\crefmultiformat{equation}{\textup{#2(#1)#3}}{ and \textup{#2(#1)#3}}
{, \textup{#2(#1)#3}}{, and \textup{#2(#1)#3}}
\crefrangemultiformat{equation}{\textup{#3(#1)#4--#5(#2)#6}}%
{ and \textup{#3(#1)#4--#5(#2)#6}}{, \textup{#3(#1)#4--#5(#2)#6}}{, and \textup{#3(#1)#4--#5(#2)#6}}

\Crefformat{equation}{#2Equation~\textup{(#1)}#3}
\Crefrangeformat{equation}{Equations~\textup{#3(#1)#4--#5(#2)#6}}
\Crefmultiformat{equation}{Equations~\textup{#2(#1)#3}}{ and \textup{#2(#1)#3}}
{, \textup{#2(#1)#3}}{, and \textup{#2(#1)#3}}
\Crefrangemultiformat{equation}{Equations~\textup{#3(#1)#4--#5(#2)#6}}%
{ and \textup{#3(#1)#4--#5(#2)#6}}{, \textup{#3(#1)#4--#5(#2)#6}}{, and \textup{#3(#1)#4--#5(#2)#6}}

\crefdefaultlabelformat{#2\textup{#1}#3}

\usepackage[left,mathlines]{lineno}
\usepackage{etoolbox} 
\newcommand*\linenomathpatchAMS[1]{%
  \expandafter\pretocmd\csname #1\endcsname {\linenomathAMS}{}{}%
  \expandafter\pretocmd\csname #1*\endcsname{\linenomathAMS}{}{}%
  \expandafter\apptocmd\csname end#1\endcsname {\endlinenomath}{}{}%
  \expandafter\apptocmd\csname end#1*\endcsname{\endlinenomath}{}{}%
}

\expandafter\ifx\linenomath\linenomathWithnumbers
  \let\linenomathAMS\linenomathWithnumbers
  \patchcmd\linenomathAMS{\advance\postdisplaypenalty\linenopenalty}{}{}{}
\else
  \let\linenomathAMS\linenomathNonumbers
\fi

\linenomathpatchAMS{gather}
\linenomathpatchAMS{multline}
\linenomathpatchAMS{align}
\linenomathpatchAMS{alignat}
\linenomathpatchAMS{flalign}

\begin{document}

\title{Circumcentric directions of cones~\thanks{RB was partially supported by the \emph{Brazilian Agency Conselho Nacional de Desenvolvimento Cient\'ifico e Tecnol\'ogico} (CNPq), Grants 304392/2018-9 and 429915/2018-7; \\ YBC was partially supported by internal funds at NIU.}}

\author{Roger Behling  \and Yunier Bello-Cruz  \and 
Hugo Lara-Urdaneta \and Harry Oviedo \and Luiz-Rafael Santos 
}

\institute{
  Roger Behling   \Letter  \at School of Applied Mathematics, Funda\c{c}\~ao Get\'ulio Vargas \\ 
Rio de Janeiro, RJ -- 22250-900, Brazil. \email{rogerbehling@gmail.com}
    \and 
Yunier Bello-Cruz \at Department of Mathematical Sciences, Northern Illinois University. \\ 
 DeKalb, IL -- 60115-2828, USA. \email{yunierbello@niu.edu}
\and 
Hugo Lara-Urdaneta  \at Department of Control Engineering and Computation, Federal University of Santa Catarina. \\ 
Blumenau, SC -- 88040-900, Brazil. \email{hugo.lara.urdaneta@ufsc.br}
 \and
 Harry Oviedo \at School of Applied Mathematics, Funda\c{c}\~ao Get\'ulio Vargas \\ 
Rio de Janeiro, RJ -- 22250-900, Brazil. \email{harry.leon@fgv.br}
\and
 Luiz-Rafael Santos \at Department of Mathematics, Federal University of Santa Catarina. \\ 
Blumenau, SC -- 88040-900, Brazil. \email{l.r.santos@ufsc.br}
}
\date{}

\titlerunning{Circumcentric directions of cones}
\authorrunning{Behling, Bello-Cruz, Lara-Urdaneta, Oviedo and Santos}
\maketitle
\vspace*{-10mm}

\begin{abstract}
Generalized circumcenters have been recently introduced and employed to speed up classical projection-type methods for solving feasibility problems. In this note, circumcenters are enforced in a new setting; they are proven to provide inward directions to sets given by convex inequalities. In particular, we show that circumcentric directions of finitely generated cones belong to the interior of their polars. We also derive a measure of interiorness of the circumcentric direction, which then provides a special cone of search directions, all being feasible to the convex region under consideration.


\end{abstract}

\section{Introduction}
\label{sec:intro}

In this paper we introduce the concept of circumcenter of a finitely generated cone $K\subset \RR^{n}$, also referred to as {\em circumcentric direction}, and we prove that it lies in $K^{\circ}$, the polar cone of $K$. This result enables us to get feasible directions to a convex set with nonempty interior of the form $\Omega\coloneqq\{x\in\RR^n \mid g(x)\le 0\}$, where $g:\RR^{n}\to \RR^{m}$ has convex and differentiable components. 

The circumcentric direction allows us to move in $\Omega$, and therefore, it might be of algorithmic interest. In fact,   the circumcentric direction   $d\in \RR^n$ with respect to $\Omega$ at a given point $\bar x\in \Omega$ is  a \emph{feasible direction to} $\Omega$  at $\bar x$, that is,  $\bar x+td\in \Omega$ for all $t>0$ sufficiently small. We will see that the circumcentric direction $d$ might be zero, but if it is nonzero, we will prove that we get an \emph{interior feasible direction}, that is also  referred to as  \emph{inward direction}. In other words, if $d\neq 0$, we have  $\bar x+td$ belonging to the interior of $\Omega$, for all small  enough $t>0$.    Our proof of this claim is derived  based on a result in \cite{Behling:2018} for generalized circumcenters. The circumcentric direction $d$ arises as the solution of a linear system of equations formed upon suitable bisectors relying on active gradients of the components of $g$. In addition to establishing that $d$ is an inward direction to $\Omega$ at the point $\bar x$, we are able to explicitly measure its level of interiorness with respect to $\Omega$. In other words, once $d$ is computed, we obtain a magnitude of how much it can be perturbed in order to still enjoy the property of pointing inwards to $\Omega$.

Circumcenters were firstly extended from elementary geometry to general Euclidean spaces in 2018 \cite{Behling:2018, Behling:2018a}. Since then, the subject of circumcenter iteration algorithms for solving feasibility problems has evolved \cite{AragonArtacho:2020, Araujo:2022, Arefidamghani:2021, Bauschke:2018, Bauschke:2020, Bauschke:2021b, Bauschke:2021d, Bauschke:2021, Behling:2020, Behling:2020a, Dizon:2022, Dizon:2022a, Lindstrom:2020, Ouyang:2018, Ouyang:2020,Ouyang:2021,Ouyang:2022,Araujo:2021,Ouyang:2021a,Behling:2021,Arefidamghani:2022,Arefidamghani:2022a}.  In all these recent references the so-called circumcentered-reflection method (CRM) was employed to speed up classical algorithms including the method of alternating projections (MAP) and the Douglas-Rachford method (DRM). The notion of circumcenter enforced in those works remains similar to the one in Euclidean geometry. For a finite set   $S \coloneqq \left\{u^{1}, u^{2}, \ldots, u^{p}\right\}$ one looks for a point  $\circum(S)$ equidistant to the points $u^i$, $i=1,\ldots,p$, and which lies on the affine subspace determined  by them, denoted here by $\aff(S)$.
We note that for an arbitrary finite set $S$, we have that $\circum(S)$ might not be well-defined. Indeed, from elementary Euclidean geometry, we know that for three distinct collinear points there is no correspondent circumcenter. On the other hand, if  $S$ comes from a round of reflections through a finite number of affine subspaces, then  \cite[Lemma 3.1]{Behling:2018} guarantees the good definition of $\circum(S)$. The good definition of circumcenters is also guaranteed when  reflections are substituted by isometries  \cite{Bauschke:2020}. Lemma 3.1 in \cite{Behling:2018} will be key for the good definition of the circumcentric direction and for deriving our results.

Having labelled the front actors of this paper allows us to have an overview of our main contributions. If $S$ is given by minus the normalized active gradients of $g$ at a point $\bar x\in \Omega$ and letting $d$ play the role of the circumcenter $\circum(S)$, what occurs is that $\bar x+td\in \Omega$, for all $t>0$ sufficiently small, as desired. If a cone $K$ is generated by the finite set $S$, we will see that $d\in K^{\circ}$. More than that, and perhaps surprisingly, $d+v$ is still in $K^{\circ}$ for all $v\in\RR^n$ satisfying $\norm{v}\leq \norm{d}^2$, where $\norm{\cdot}$ denotes the Euclidean norm. The notation $\scal{\cdot}{\cdot}$ is for the scalar product, which induces the Euclidean norm.  These results are formalized in the next section. \Cref{sec:concluding} discusses our contributions by means of examples and propositions connecting pointed cones with nonzero circumcenters. To conclude our note, at the end of \Cref{sec:concluding}, we propose some ideas for future work.

\section{Circumcentric directions}\label{sec:circumdir}

Before addressing the actual content of our work, we present some basic definitions.

\begin{definition}[polar and dual of a set]\label[definition]{def:polar_dual} Let $X\subset \RR^{n}$. The \emph{polar} of  $X$ is a set of $\RR^{n}$ defined by
\[
X^{\circ} \coloneqq \{w\in \RR^{n}\mid \scal{z}{w}\le 0, \forall z\in X\}.
\]
 The \emph{dual} of $X$ is given by
\(
X^{*} \coloneqq -X^{\circ} 
\).

\end{definition}
 
Of course,  $X^{\circ}$ and $X^{*}$  are nonempty for the  zero vector lies in  them. Next, let us recall some notions regarding cones.

\begin{definition}[cone] \label[definition]{def:cone}
A subset $K$ of $\RR^{n}$ is called a \emph{cone} if it is closed under positive scalar multiplication, \emph{i.e.}, $t w \in K$ when $w \in K$ and $t>0$.  A \emph{convex cone} is a cone that is a convex set.

\end{definition}

\Cref{def:cone,def:polar_dual} clearly give us  that for any set $X\subset\RR^{n}$ both $X^{\circ}$ and $X^{*}$ are always cones that are closed and convex.

\begin{definition}[pointed cone]
  A cone $K\subset \RR^n$ is said to be \emph{pointed} if $K\cap (-K) = \{0\}.$
  \end{definition}
  
There are equivalent statements for what is understood as a pointed cone $K$. Two of them are: $ K$ contains no lines \cite[Thm.~4.12]{Bertsimas:1997}; and $\inte(\Kpolar)$ is nonempty  \cite[p.~213, 3.6(d)]{Bertsekas:2003} and \cite[Prop.~2]{Studeny:1993}.

For a finite number of vectors $w^{1}, \ldots, w^{p}\in \RR^n$, a sum of the form  $\sum_{i=1}^{p} \lambda_{i} w^{i}$ is said to be a \emph{conic combination} of $w^{1}, \ldots, w^{p}$ if the scalars $\lambda_{1}, \ldots, \lambda_{p}$ are all nonnegative. 
The \emph{cone generated by $X\subset \RR^n,$} denoted by $\cone(X)$, is the set of all conic combinations of elements of $X$. Note that $\cone(X)$ is convex and, if $X$ is nonempty,  $\cone(X)$ contains the zero vector.

\begin{definition}[finitely generated cone]
A cone $K$ of $\RR^{n}$ is said to be finitely generated if there exists a finite set $S\subset \RR^{n}$ such that $0\notin S$ and $\cone(S) = K$. A \emph{conic base} of a finitely generated cone $K$ is a finite set $B_{K}\subset\RR^{n}$ with minimal cardinality such that $\cone(B_{K}) = K$.

\end{definition}

In addition to being convex, a finitely generated cone is also closed.  The definitions and remarks above can be found, for instance, in \cite{Bertsekas:2003}. 

We recall that a set $\Omega\subset \RR^{n}$ is  a (convex) \emph{polyhedron}, if it can be expressed as the intersection of a finite family of closed half-spaces, that is, 
\[\label{eq:polyhedron}
  \Omega \coloneqq \{x \in \RR^{n} \mid \scal{a^i}{x} \leq b_i, a^{i}\in\RR^{n}, b_{i}\in\RR,  \text{ for } i=1,\ldots, m\}.
\]
With that said, we can define a \emph{polyhedral cone}, which is a set that is simultaneously a cone and a polyhedron. It is well known that a cone is polyhedral if, and only if, it is finitely generated. This can be seen as a corollary of the  Minkowski-Weyl theorem; see, for instance,~\cite[Thm.~3.52]{Rockafellar:2004}. 

Another important concept regarding cones is the one of tangent cones. 

\begin{definition}[{Tangent cone \cite[Prop.~A.5.2.1]{Hiriart-Urruty:2001}}]\label{def:tangent_cone} 
  Let $X$  be a nonempty closed convex set in $\RR^{n}$ and $ x\in X$. The \emph{tangent cone} of  $X$ at $ x$ is given by 
    \[
      \Tcone_{X}( x)\coloneqq \closu\left(\left\{ \lambda(y- x) \in \RR^{n}\mid y \in X, \lambda \geq 0 \right \}\right).
    \]
    \end{definition}

The next result states that a tangent cone of a polyhedron is also a polyhedron.

\begin{fact}[Tangent cone of polyhedron {\cite[Thm. 6.46]{Rockafellar:2004}}]\label[fact]{fact:tangent-cone-poly}
  Let $\Omega\subset \RR^{n}$ be a polyhedron defined as in \cref{eq:polyhedron}.  Then,   the tangent cone $\Tcone_{\Omega}(\bar x)$, at any point $\bar x\in \Omega$, is a polyhedral cone and can be represented as 
  \[
  \Tcone_{\Omega}(\bar x) = \{w \in \RR^{n} \mid \scal{a^i}{w}\le 0, \text{ for } i \in J(\bar x)\},
  \]
  where  $J(\bar x) \coloneqq\{i  \mid \scal{a^i}{\bar x} = b_i \}$ is the active index set of $\Omega$ at $\bar x$.
  \end{fact}

We now define the circumcenter of a finite set $S\subset \RR^n$. Remind that $\aff(S)$ stands for the smallest affine subspace containing $S$.

\begin{definition}[circumcenter of $p$ points]
Let $S\coloneqq \left\{u^{1}, u^{2}, \ldots, u^{p}\right\},$ where $m$ is a positive integer  and $u^{1}, u^{2}, \ldots, u^{p}$ are in $\RR^{n}$. A vector $d\in \RR^{n}$ is called the \emph{circumcenter of $S$} if it satisfies the following two conditions:

\begin{listi}
	\item $d \in \aff(S)$, and
	\item $\left\|d-u^{1}\right\|=\left\|d-u^{2}\right\|=\cdots=$
$\left\|d-u^{p}\right\|$.
\end{listi}
 \end{definition} 

Our notation for  a generalized circumcenter of a set   $S \coloneqq \left\{u^{1}, u^{2}, \ldots, u^{p}\right\}$ is $\circum(S)$. Clearly, if $S=\{u^{1}\}$, then $\circum(S) = u^{1}$ and if $S=\{u^{1},u^{2}\}$, we have $\circum(S) = \frac{1}{2}(u^{1} + u^{2})$. In the cases where $p>2$, $\circum(S)$ may not exist; see a discussion on properness of generalized circumcenters in  \cite{Bauschke:2018}. If $\circum(S)$ exists, it must be unique \cite[Prop.~3.3]{Bauschke:2018}  and reads as 
\[\label{eq:CircComputation}
\circum(S)=u^{1}+\alpha_1\left(u^{2}-u^{1}\right) + \cdots + \alpha_{p-1}\left(u^{p}-u^{1}\right), 
\]
where $(\alpha_1,\ldots,\alpha_{p-1})\in \RR^{p-1}$ is any  solution of the $p-1\times p-1$ linear system of equations  whose $i$-th row is given by 
\[\label{eq:CircMatrix}
  \sum_{j=1}^{p-1} \alpha_{j}\left\langle u^{j+1}-u^{1}, u^{i+1}-u^{1}\right\rangle=\frac{1}{2}\left\|u^{i+1}-u^{1}\right\|^{2}.
\]
The $p-1$ equations  \cref{eq:CircMatrix} determine a Gram matrix~\cite{Behling:2018a,Bauschke:2018}.


We will see next that a sufficient condition for $\circum(S)$ to exist is that all vectors in $S$ have the same length. Moreover, in this case $\circum(S)$   will be characterized as the orthogonal projection of the origin onto $\aff\left(S\right)$. 

Recall that the orthogonal projection of a point $y\in \RR^{n}$ onto a closed and convex set $X\subset \RR^{n}$  is given  by  $P_{X}(y)\in X$ if, and only if, $\scal{z-P_{X}(y)}{y - P_{X}(y)}\le 0,$ for all $z\in X$. This is, of course, equivalent to   \(P_{X}(y) \coloneqq \argmin_{z\in X} \{\norm{y - z } \}.\)

\begin{lemma}[characterization of circumcenters] \label[lemma]{lema:circum}
Let $S\coloneqq \{u^{1}, u^{2}, \ldots, u^{p}\}\subset\RR^{n},$ where $p$ is a positive integer, and assume that all the vectors $u^i$, $i=1,\ldots,p$, have the same length $\eta\geq 0$. Then, the circumcenter of $S$, $\circum(S)$ is proper and, moreover,
\begin{listi}
  \item $\circum(S) = P_{\aff(S)}(0)$;
  \item $\scal{\circum(S)}{u^i}= \norm{\circum(S)}^2  \geq 0$  for all $i=1,\ldots,p$;
  \item If $\eta>0$, then  $\scal{\circum(S)+v}{u^i} \geq 0,$ for all $v\in\RR^{n}$ such that $\norm{v}\leq \frac{\norm{\circum(S)}^2}{\eta}$ and $i=1,\ldots,p$.
\end{listi}
\end{lemma}

\begin{proof}
    Note that by proving (i), we automatically get the well-definedness of the circumcenter. So, let us proceed with the proof of this item. The idea is to employ Lemma 3.1 from \cite{Behling:2018}, which has the Pythagorean theorem at its core. For this, we are going to define the suitable subspaces $V_1,\ldots, V_{p-1}$, where each $V_j$ is the line connecting the origin with $\frac{u^j + u^{j+1}}{2}$, where $j=1,\ldots, p-1$. Then, due to the fact that $u^1,u^2,\ldots, u^p$ have the same length, it is easy to see that $R_{V_j}(u^j) =   u^{j+1}$, for all $j\in\{1,2,\ldots, p-1\}$. Indeed, 
    \begin{align}
      R_{V_j}(u^j) & = 2 P_{V_j}(u^j) - u^j  = 2\left(\frac{u^j + u^{j+1}}{2}\right) - u^j  = u^{j+1},
    \end{align}
    where the second equality follows from properties  of projections onto subspaces, as \[\scal{u^j - \frac{u^j + u^{j+1}}{2}}{\frac{u^j + u^{j+1}}{2}} =  \scal{ \frac{u^j - u^{j+1}}{2}}{\frac{u^j + u^{j+1}}{2}} = \frac{1}{4} \left(\norm{u^{j}}^2 - \norm{u^{j+1}}^2\right) = \frac{1}{4} \left(\eta^2 - \eta^2 \right) = 0.\]
    Since $0\in \cap_{i=1}^{p} V_i$, Lemma 3.1 from \cite{Behling:2018} directly provides $\circum(S) = P_{\aff(S)}(0)$, which proves item (i).
    
    From the characterization of projections onto affine sets, we have, for all $z\in \aff(S)$, that
    \begin{equation}\label{eq_affine}
    \scal{z - \circum(S)}{0 - \circum(S)} = 0, 
    \end{equation}
  which implies that 
    \[\scal{\circum(S)}{z} = \norm{\circum(S)}^2 \geq 0.\]
    Since all $u^i$'s are in $\aff(S)$, item (ii) follows. 
    
    Item (iii) is a consequence of the manipulations presented below. Let $v$ be any vector in $\RR^n$ satisfying $\norm{v}\leq \frac{\norm{\circum(S)}^2}{\eta}$. Then, 
    \begin{align}
    \scal{\circum(S)+v}{u^i} 
                             & =  \scal{\circum(S)}{u^i - \circum(S) + \circum(S)} + \scal{v}{u^i}  \\
                             & =  \scal{\circum(S)}{u^i - \circum(S)} + \norm{\circum(S)}^2 + \scal{v}{u^i}  \\
                             & =  \norm{\circum(S)}^2 + \scal{v}{u^i}  \\
                             & \geq   \norm{\circum(S)}^2 - \norm{v}\norm{u^i}  \\
                             & =  \norm{\circum(S)}^2 - \norm{v}\eta.  \\
                             & \geq  \norm{\circum(S)}^2 - \frac{\norm{\circum(S)}^2}{\eta}\eta  \\ 
                             & =  0.
    \end{align}
    The first two equalities above follow from inner product properties. The third one is by item (ii). The first inequality is due to Cauchy-Schwarz and the fourth equality comes from the definition of $\eta$. The second inequality is obtained by employing the hypothesis on the norm of $v$, which then yields the result. 
\end{proof}

  The previous lemma is key for developing the theory of circumcenters of cones,
  which begins with the following definition.

\begin{definition}[circumcentric direction of polyhedral cones]\label[definition]{def:circdirection}
If $K\subset \RR^{n}$ is a nontrivial polyhedral cone and  $B_{K}\coloneqq \{u^{1},\ldots,u^{p}\}$ is a normalized conic base of $K$, \emph{i.e.}, a conic base whose vectors have all norm equal to one, we say that $d  \coloneqq -\circum(B_{K})$ is the \emph{circumcentric direction} of $K$. If $K=\{0\}$, the \emph{circumcentric direction} $d$ is zero. We also use the notation $d_{\circum}(K)$ for $d$.
\end{definition}

The next proposition establishes the well-definedness of the circumcenter of a polyhedral cone.

\begin{proposition}[good definition of circumcentric direction of polyhedral cones] \label[proposition]{prop:conic_base}
Having  a polyhedral cone $K\subset \RR^{n}$, its circumcentric direction exists and is unique.  
\end{proposition}

\begin{proof}
  If $K = \{0\}$, the circumcentric direction is zero by definition. So, let $K$ be nontrivial. 
  Suppose first that $K$ is not pointed.  Then, there exists  $x\in K\cap(-K)$ such that $x\neq 0$, and thus we get
  \[\label{eq:prop_xpos}
    x = \sum_{j=1}^{p}  \alpha_ju^j, \alpha_j \geq 0, j=1\ldots,p,
  \]
with $\sum_{j=1}^{p}  \alpha_j >0$. 
  Moreover, we have that $-x \in K$, that is, 
  \[\label{eq:prop_xneg}
    -x = \sum_{j=1}^{p}  \beta_ju^j, \beta_j \geq 0, j=1\ldots,p,
  \]
  where $\sum_{j=1}^{p}  \beta_j >0$. Summing \cref{eq:prop_xpos} and 
\cref{eq:prop_xneg} and dividing by $\sum_{j=1}^{p} (\alpha_j +  \beta_j)>0$ we have  
\[
0 = \frac{\sum_{j=1}^{p} (\alpha_j +  \beta_j) u^j}{\sum_{j=1}^{p} (\alpha_j +  \beta_j)}.
\]
Because the coefficients on the right-hand side of the last equation add up to $1$, we get that $0$ belongs to $\aff(B_K)$, and  \Cref{lema:circum}(i) yields $\circum(B_K) = P_{\aff(B_K)}(0) = 0$.

Assume now that $K$ is pointed  and let $B_{K}\coloneqq \{u^{1},\ldots,u^{p}\}$ be a normalized conic base of $K$.  
We remark that what we call here  a conic base is referred to in \cite[pg. 176]{Bertsimas:1997}  as a \emph{complete set of extreme rays}, and thus  every element of a conic base must be a  generator of an extreme ray; see \cite[Def.~4.2]{Bertsimas:1997}.  Then, there exists only one set $B_K$ that we can call \emph{normalized conic base} of $K$, that is, a conic base where  all its elements  have norm equal to $1$. Thus, uniqueness of the circumcentric direction  is guaranteed.
The existence is due to \Cref{lema:circum}(i).   
\end{proof}

We just have seen that the circumcentric direction of a polyhedral cone is  well-defined. 
Note further that  the circumcentric direction must have size of at most $1$ because we ask for normalized generators in the definition. 
In fact, what really matters for the good definition of the circumcentric direction is that the generators have all the same size. We take them with norm one just for simplicity. 

We proceed next by showing that the circumcentric direction $d$ of a polyhedral cone belongs to its polar. In addition to that, we provide a ball centered in $d$ with radius $\norm{d}^2$ whose points still lie in the polar. This elegant measure of interiorness is stated in the following theorem.

\begin{theorem}[properties of circumcentric directions of polyhedral cones]\label{Thm:circumcentric_d_cones}
Let $K\subset \RR^{n}$ be a polyhedral cone and $d\in \RR^{n}$ its circumcentric direction. Then, $d+v\in \Kpolar$, for all $v\in\RR^{n}$ such that $\norm{v}\leq \norm{d}^2$ and, in particular, $d\in \Kpolar$. Moreover, if $d\neq 0$ then $d+v \in \inte (\Kpolar)$, for all $v\in\RR^{n}$, whenever  $\norm{v}< \norm{d}^2$.
\end{theorem}
\begin{proof}
Let $B_{K} \coloneqq \{u^1,\ldots,u^p\}$ be a normalized conic base of $K$ and set $d\coloneqq - \circum(B_{K})$. So, item (iii) of \Cref{lema:circum} applies with $B_{K}$ playing the role of $S$ and $\eta=1$. Therefore,  $\scal{d+v}{u^i}\le 0$ for all $i=1,\ldots,p$, if $\|v\|\le \|d\|^2$. For any $z\in K$, we get $\scal{d+v}{z}\le 0$ because $z$ is a conic combination of the $u^i$'s. Hence, $d+v\in \Kpolar$ and, in particular, $d \in \Kpolar$.  Now, if $d\neq0$ and $\|v\|< \|d\|^2$, we have
$  \scal{d+v}{u^i}   \leq   \norm{v} - \norm{d}^2  < 0$, for all $i=1,\dots,p$. So, $\scal{d+v}{z}<0$, for all $z\in K$, which implies  that  $d+v \in  \inte(\Kpolar)$. 
\end{proof}

The theorem we just established can be used to get inward directions for convex regions. In this regard, we start providing feasible directions for polyhedral sets. 

\begin{corollary}\label[corollary]{cor:Ax le b}
Let $\Omega=\{x\in\RR^n \mid Ax\le b\}$ where $b\in\RR^m$, $A\in \RR^{m\times n}$ has nonzero rows $a^1,\ldots,a^m$, and $\bar x$ be a given point in $\Omega$. Define \[J(\bar x)\coloneqq\{j\in\{1,\ldots,m\}\mid \scal{a^j}{\bar x}=b_j\} \quad \mbox{and} \quad d\coloneqq -\circum\left(\cone (\{a^j\}_{j\in J(\bar x)})\right).\] Then, there exists $\delta>0$ such that $\bar x+t(d+v)\in\Omega$ for all $t\in[0,\delta]$ and $v\in \RR^n$ satisfying $\|v\|\le \|d\|^2$. 
\end{corollary}
\begin{proof}
By setting $K\coloneqq\cone (\{a^j\}_{j\in J(\bar x)})$ we get from \Cref{Thm:circumcentric_d_cones} that $d+v\in \Kpolar$ with $\|v\|\le \|d\|^2$. Now, it is well-known (see \cite[Fact 2.9]{Behling:2021a}) that the polyhedron $\Omega$ coincides locally with its tangent cone at $\bar x$, that is, there exists $\delta>0$ such that $\Omega\cap B_\delta(\bar x)=(\bar x+\Tcone_{\Omega}(\bar x))\cap B_\delta(\bar x)$. Furthermore,  because of \Cref{def:polar_dual,fact:tangent-cone-poly},  $\Tcone_{\Omega}(\bar x)$  is precisely $\Kpolar$. Thus, we have $\Omega\cap B_\delta(\bar x)=(\bar x+\Kpolar)\cap B_\delta(\bar x)$ and the result follows. 
\end{proof}
We now employ \Cref{Thm:circumcentric_d_cones} to derive feasible directions upon the circumcentric direction for more general convex sets.
\begin{corollary}\label[corollary]{cor:g(x) le 0}
Let $\Omega=\{x\in\RR^n \mid g(x)\le 0\}$ where $g:\RR^n\to \RR^m$ has convex and differentiable components, and $\bar x$ be a given point in $\Omega$. Assume also that $\Omega$ satisfies the Slater condition, \emph{i.e.}, there exists $\hat x$ such that $g(\hat x)<0$. Define $J(\bar x)\coloneqq\{j\in\{1,\ldots,m\}\mid g_j(\bar x)=0\}$ and $d=-\circum\left(\cone (\{\nabla g_j(\bar x)\}_{j\in J(\bar x)})\right)$. Then, for each $v\in \RR^n$ satisfying $\|v\|< \|d\|^2$, there exists $\delta_v>0$ such that $\bar x+t(d+v)\in\Omega$ for all $t\in[0,\delta_v]$. 
\end{corollary}
\begin{proof}
Set $K:=\cone (\{\nabla g_j(\bar x)\}_{j\in J(\bar x)})$. If $d:=d_{\circum}(K)$ is zero, then the result is trivial, so, assume $d\neq 0$. As Slater condition is a constrained qualification, it guarantees that $\Kpolar$ coincides with the tangent cone of $\Omega$ at $\bar x$ and also that $\Kpolar$ has nonempty interior; see condition CQ5c and subsequent comments in  \cite[pg.~307]{Bertsekas:2003} . Hence, any direction in the interior of $\Kpolar$, $\inte(\Kpolar)$, is a feasible direction for $\Omega$. On the other hand,  \Cref{Thm:circumcentric_d_cones} implies that $d+v\in \inte(\Kpolar)$ if $\|v\|< \|d\|^2$. Thus, $d+v$ is an interior feasible direction for $\Omega$ at $\bar x$, and the result follows.
\end{proof}

Getting inward directions to a convex region is a recurrent task in several optimization problems. We have derived a novel and quite straightforward manner to do this by introducing circumcentric directions of finitely generated cones. 

The computation of circumcentric directions both in the polyhedral case of \Cref{cor:Ax le b}   and in the more general convex setting of  \Cref{cor:g(x) le 0} relies on two tasks: obtaining a conic base of a suitable polyhedral cone; and solving the linear system of equations \cref{eq:CircComputation}. Active equations, as mentioned in the aforementioned  results, provide generators for the desired cone. Nevertheless, the first task  might be challenging  in the presence of  redundancy or degeneracy. Actually, getting rid of redundant or degenerated constraints can be  computationally expensive; for more details, see \cite[Sec.~4.5]{Bazaraa:2009}. Once a conic base is available, concluding the second task is easy as it involves a well understood solvable linear system of equations with  unique solution \cite{Bauschke:2018}.    

In the next section, we present some remarks on our theory.

\section{Discussion of results}\label{sec:concluding}

We begin this section by pointing out the importance of considering a conic base in the definition of the circumcentric  for it to possess a genuine \emph{geometric} characterization. If one considers a set of generators instead of a base we may have ambiguity. Indeed, see the example below.
\begin{example}\label[example]{eq:1}
  Let $K\subset\RR^{3}$ be the cone generated by \[S \coloneqq \left\{(1,0,0), (0,1,0), (0,0,1), (\tfrac{\sqrt{2}}{2},\tfrac{\sqrt{2}}{2},0)\right\}.\]
We have that $K = \cone(S)  =  \RR^{3}_{+}$ and the  only normalized conic base of $K$ is given by \[B_{K} = \{(1,0,0), (0,1,0), (0,0,1)\}.
\]
However, the circumcentric direction of $K$ satisfies \[d=d_{\circum}(K)= -\circum(B_{K}) = -\tfrac{1}{3}\left(1,1,1\right)\neq (0,0,0) = -\circum(S).\]
\end{example}

We have seen in this paper that $d$ always lies in $\Kpolar$. In the previous example $-d$ coincidentally belongs to $K$, and it is the barycenter of the convex hull of $B_K$. Nevertheless, these two properties do not hold in general, as pointed out later in \Cref{ex:2}

Our next remark is on the Slater assumption in \Cref{cor:g(x) le 0}. It is well-known that the Slater condition is equivalent to the Mangasarian-Fromovitz constraint qualification (MFCQ) for convex inequalities~\cite[Prop. 3.3.8 and 3.3.9]{Bertsekas:1999}. In turn, MFCQ is equivalent to asking for positive linear independence of active gradients. With those facts in mind and taking into account the following proposition, we necessarily would have a zero circumcentric direction in \Cref{cor:g(x) le 0} without the existence of a Slater point.   

\begin{proposition}\label[proposition]{prop:circdirnot0}
Let $S\coloneqq \{u^{1}, u^{2}, \ldots, u^{p}\}\subset\RR^{n},$ where $p$ is a positive integer, and assume that all the vectors $u^i$, $i=1,\ldots,p$, have the same length $\eta\geq 0$. If $S$ is not positively linearly independent, then $\circum(S)= 0$.  
\end{proposition}
\begin{proof}
If the length of the vectors is zero, the result follows trivially. So, let us assume that their length is positive. Recall that $\circum(S)$ is well defined and characterized as $\circum(S) = P_{\aff(S)}(0)$ due to \Cref{lema:circum}(i). Suppose that $S$ is not positively linearly independent. Then, there exist nonnegative scalars $\alpha_1,\ldots,\alpha_p$ where at least one of them is strictly positive such that $\alpha_1 u^{1}+\cdots + \alpha_p u^{p}=0$.
Since, $\sum_{i=1}^p \alpha_i>0$, we can divide the previous equality by this sum getting
\[\frac{\alpha_1}{\sum_{i=1}^p \alpha_i} u^{1}+\cdots + \frac{\alpha_p}{\sum_{i=1}^p \alpha_i} u^{p}=0.\] By setting $\beta_j\coloneqq\frac{\alpha_j}{\sum_{i=1}^p\alpha_i}$, we clearly have $\beta_j\ge0$ for $j=1,\ldots,p$ and $\sum_{j=1}^p\beta_j=1$. Hence, $0$ belongs to $\aff(S)$ and thus $\circum(S) = P_{\aff(S)}(0)=0$. 
\end{proof}
We observe that the converse of  \Cref{prop:circdirnot0} is not true. Take, for instance, $S=\{v,-v\}$ where $v$ is an arbitrary unit vector in $\RR^n$. Obviously, $\circum(S)=0$. Nonetheless, $v$ and $-v$ are positively linearly independent. Next, we prove that if $\circum(S)\neq 0$ then $\cone(S)$ is pointed.


\begin{proposition}\label[proposition]{prop:conebase-dcirc} 
Let $K\subset \RR^n$ be a finitely generated cone. If $d_{\circum}(K)\neq 0$, then $K$ is pointed.  
\end{proposition}
\begin{proof}

If $d_{\circum}(K)\neq 0$, \Cref{Thm:circumcentric_d_cones} guarantees that $-d_{\circum}(K)\in \inte(\Kpolar)$. Hence, $\inte(\Kpolar)$ is nonempty. Bearing in mind that a closed convex cone is pointed if, and only if, its polar has nonempty interior, we have that $K$ is pointed.    
\end{proof}


We show in the following example that the converse of \Cref{prop:conebase-dcirc} does not hold. 

\begin{example}\label[example]{ex:2}
Take
\[u^1\coloneqq (0,\tfrac{\sqrt{2}}{2},\tfrac{\sqrt{2}}{2}), \quad u^2\coloneqq (\tfrac{\sqrt{2}}{2},0,\tfrac{\sqrt{2}}{2}),\quad u^3\coloneqq (-\tfrac{1}{2},\tfrac{1}{2},\tfrac{\sqrt{2}}{2}),\] and set $S\coloneqq \{u^1,u^2,u^3\}$. It is easy to verify that $d\coloneqq\circum(S)=(0,0,\tfrac{\sqrt{2}}{2})$, and  that $d$ does not lie in the convex hull of $S$. This  implies that the set of unit vectors $\bar S\coloneqq \{u^1, u^2, u^3, u^4\}$ with $u^4\coloneqq\tfrac{d}{\|d\|}=(0,0,1)$, is positive linearly independent. Therefore, $\bar S$ form a conic base for $\cone(\bar S)$. However, despite $\cone(\bar S)$ being pointed, $\circum(\bar S)=0$. 
\end{example}

Although the reciprocal of \Cref{prop:conebase-dcirc} is not true, we have a special class of pointed cones for which the circumcentric direction must be nonzero. The next result will lead to a corollary providing  this class. 

\begin{proposition}\label[proposition]{prop:linearInd-base} 
Let $S=\{u^1,\ldots,u^p\}$ be a set of linearly independent unit vectors. Then,  $\circum(S)\neq 0$. In particular, $\cone(S)$ is pointed.  
\end{proposition}
\begin{proof}
Suppose that $\circum(S)= 0$. Then, \Cref{lema:circum}(i) implies $0\in\aff(S)$, which means that there exist scalars $\beta_1,\ldots, \beta_p$ such that $\sum_{i=1}^p\beta_i=1$ and
\[0=\beta_1u^1+\cdots+\beta_pu^p.\] This contradicts the linear independence of $u^1,\ldots,u^p$. Hence, $\circum(S)\neq 0$, which guarantees that $\cone(S)$ is pointed in view of \Cref{prop:conebase-dcirc}.
\end{proof}

\begin{corollary}\label[corollary]{prop:linearInd-pointed} 
Let $K\subset \RR^n$ be a finitely generated cone, $B_K=\{u^1,\ldots,u^p\}$ be a unitary conic base of $K$ and assume that $B_K$ is linearly independent. Then, $K$ is pointed  and $d_{\circum}(K)\neq 0$.   
\end{corollary}


In this paper, we have embedded the circumcenter in a new setting. Generalized circumcenters serve now as inward directions to convex sets.  The fact that circumcentric directions point towards the interior of a given convex set together with our explicit measure of its interiorness are very attractive features that may have an impact in practical algorithmic implementations. In this regard, one of our ideas for the future is to explore circumcenters as search directions within methods for both smooth and nonsmooth convex optimization. Moreover, based on \Cref{lema:circum}(i) we intend to extend the notion of circumcenters of finitely generated cones to more general cones.

\bibliographystyle{spmpsci}

\bibliography{refs}

\end{document}